\documentclass[a4paper,11pt]{article}
\usepackage[utf8]{inputenc}

\usepackage{amssymb}
\usepackage{amsmath}
\usepackage{amsfonts}
\usepackage{amsthm}
\usepackage{url}

\newcommand{\Co}{\mathcal{C}}
\newcommand{\R}{\mathbb{R}}
\newcommand{\ZZ}{\mathbb{Z}}
\newcommand{\KK}{\mathbb{K}}
\newcommand{\NN}{\mathbb{N}}
\newcommand{\F}{\mathbb{F}}

\newcommand{\Gr}{\text{Gr}}

\newcommand{\bx}{\mathbf{x}}

\newcommand{\lt}{\text{lt}}
\newcommand{\supp}{\text{supp}}
\newcommand{\wt}{\text{wt}}

\newcommand{\rank}{\text{rk}}

\newtheoremstyle{theorem}
  {10pt}		  
  {10pt}  
  {\sl}  
  {\parindent}     
  {\bf}  
  {. }    
  { }    
  {}     
\theoremstyle{theorem}
\newtheorem{theorem}{Theorem}
\newtheorem{corollary}[theorem]{Corollary}

\newtheoremstyle{defi}
  {10pt}		  
  {10pt}  
  {\rm}  
  {\parindent}     
  {\bf}  
  {. }    
  { }    
  {}     
\theoremstyle{defi}

\newcounter{exampleNo}

\newtheorem{proposition}[theorem]{Proposition}

\newenvironment{example}[1][Example \arabic{exampleNo}.]{\begin{trivlist}\refstepcounter{exampleNo}
\item[\hskip \labelsep {\bfseries #1}]}{\end{trivlist}}

\title{Universal Gröbner Bases for Binary Linear Codes}
\author{Natalia Dück$^1$ and Karl-Heinz Zimmermann$^2$\\
{\small $^{1,2}$ Hamburg University of Technology}\\ {\small 21073 Hamburg, GERMANY}\\[2pt]}

\date{}

\begin{document}

\maketitle
\begin{abstract}
Each linear code can be described by a code ideal given as the sum of a toric ideal and a non-prime ideal. 
In this way, several concepts from the theory of toric ideals can be translated into the setting of code ideals.
It will be shown that after adjusting some of these concepts, 
the same inclusion relationship between the set of circuits, the universal Gröbner basis and the Graver basis holds. 
Furthermore, in the case of binary linear codes,  the universal Gröbner basis will consist of all binomials 
which correspond to codewords that satisfy the Singleton bound and a particular rank condition.
This will give rise to a new class of binary linear codes denoted as Singleton codes.
\end{abstract}



\section{Introduction}
When digital data are transmitted through a noisy channel errors can occur. 
But as receiving correct data is indispensable in many applications, error-correcting codes, 
which allow the detection and correction of a certain amount of errors by adding redundancy, 
are employed to tackle this problem. 
The construction of such codes and the study of their key properties is an ongoing and important task.

Gröbner bases, on the other hand, are a powerful tool that has originated from commutative algebra 
and provide a uniform approach to grasping a wide range of problems such as solving algebraic systems of equations, 
ideal membership decision, and effective computation in residue class rings modulo polynomial ideals~\cite{adams,becker,cls,sturmfels}.

The first connection between linear codes and Gröbner bases was established in~\cite{cooper} which soon became known as the ''Cooper philosophy``.
This link was based on the description of cyclic codes as ideals in a certain polynomial ring, where entries of a codeword
are viewed as coefficients of a polynomial.

In~\cite{borges}, a different connection between linear codes and ideals in polynomial rings was presented, 
which was followed up in \cite{mahkhz4, mahkhz3}.
In this approach, linear codes are described by a binomial ideal in a polynomial ring over an arbitrary field that can be written
as the sum of a toric ideal and a non-prime ideal, the so-called code ideal. 
Toric ideals play a central role in this framework and are a well-studied special class of ideals in polynomial rings arising in various 
applications~\cite{robbiano,contitrav-intprogramm,fulton, sturmfels}.
This correspondence with linear codes proved to be extremely beneficial as it allowed the application of (slightly modified) results from the rich 
theory of toric ideals~\cite{mahkhz3}. 
Furthermore, it has been shown that it allows an easy construction of the reduced Gröbner basis with respect to the lexicographic order
from a systematic generator matrix~\cite{mahkhz4}. 
Many results in algebraic geometry, however, depend on the chosen monomial order and thus
knowledge of the universal Gröbner basis for this ideal, i.e., a finite generating set of the ideal that is a Gröbner
basis for all monomial orders, is advantageous~\cite{weis}.

In this paper, some concepts used in connection with toric ideals are adapted to code ideals.
It will be shown that for any code ideal the same inclusion relationship between the set of circuits, 
the universal Gröbner basis and the Graver basis holds as for toric ideals. 
In particular, for binary codes, all three sets will coincide. 
Furthermore, it will be proved that the universal Gröbner basis for a binary linear code has a neat structure: it
consists of all binomials associated to codewords which satisfy the Singleton bound and a particular rank condition.
This gives rise to a new class of codes denoted as Singleton codes.

This paper is organized as follows. The next section presents the basics about Gröbner bases and introduces 
toric ideals. In the third section, the main notions of linear codes are facilitated and their connection 
to ideals is established. The main results are contained in the fourth section.


\section{Gröbner Bases and Toric Ideals}
Throughout the paper, let $\KK$ denote a field, $\ZZ$ the set of integers, and $\NN_0$ the set of non-negative integers.
Denote by $\KK[\bx]\!=\!\KK[x_1,\dots,x_n]$ the polynomial ring in $n$ indeterminates $x_1,\ldots,x_n$ and
by $\bx^{u} = x_1^{u_1}x_2^{u_2}\cdots x_n^{u_n}$ 
the \textit{monomials\/} in $\KK[\bx]$, where $u=(u_1,\ldots,u_n)\in\NN_0^n$.
The {\em total degree\/} of a monomial $\bx^{u}$ in $\KK[\bx]$ is given by the sum of exponents $|u|=u_1+\ldots+u_n$.
For any $\omega\in\R^n$, the $\omega$-degree of $\bx^u$ is defined by the usual inner product $u\cdot\omega$ in $\R^n$.

A \textit{monomial order\/} on $\KK[\bx]$ is a relation $\succ$ on the set of monomials in $\KK[\bx]$ satisfying:
(1) $\succ$ is a total ordering, (2) the monomial $\bx^\mathbf{0}=1$ is the unique minimal element, and
(3) $\bx^u\succ\bx^v$ implies $\bx^u\bx^w\succ\bx^v\bx^w$ for all $u,v,w\in\NN_0^n$.

Prominent monomial orders are the lexicographic order, the degree lexicographic order, and the degree reverse lexicographic order.

Given a monomial order $\succ$, each non-zero polynomial $f\in\KK[\bx]$ has a unique \textit{leading term}, denoted by $\lt_\succ(f)$ or simply $\lt(f)$, 
which is given by the largest involved term. 
The coefficient and the monomial of the leading term are called 
the \textit{leading coefficient\/} and the \textit{leading monomial}, respectively.

If $I$ is an ideal in $\KK[\bx]$ and $\succ$ is a monomial order on $\KK[\bx]$, 
its \textit{leading ideal\/} is the monomial ideal generated by the leading monomials of its elements,
\begin{align}
\langle \lt(I)\rangle = \langle\lt(f)\mid f\in I\rangle.
\end{align}
A finite subset $\mathcal{G}$ of an ideal $I$ in $\KK[\bx]$ is a {\em Gröbner basis\/} for $I$ with respect to $\succ$ 
if the leading ideal of $I$ is generated by the set of leading monomials in~$\mathcal{G}$; 
that is, 
\begin{align}
\langle\lt(I)\rangle = \langle\lt(g)\mid g\in \mathcal{G}\rangle.
\end{align}
A  Gröbner basis $\mathcal{G}$ is {\em minimal} if no monomial in the set $\mathcal{G}$ is redundant, 
and it is {\em reduced\/} if for any two distinct elements $g,h\in \mathcal{G}$, no term of $h$ is divisible by $\lt(g)$
and all its elements have leading coefficient equal to~1.
A reduced Gröbner basis for an ideal $I$ with respect to a monomial order~$\succ$ is uniquely determined
and will be denoted by $\mathcal{G}_\succ(I)$.

Gröbner bases solve the ideal membership problem. More concretely, a polynomial $f$ belongs to an ideal $I$ if and
only if it is being reduced to zero on division by a Gröbner basis for $I$. 

Gröbner bases can by computed by \textit{Buchberger's algorithm} which is implemented in most computer algebra systems.

Although infinitely many monomial orders exist, there are only finitely many reduced Gröbner bases for a given ideal. 
The union of all reduced Gröbner bases is called the \textit{universal Gröbner basis} for $I$.
More details on Gröbner bases can be found in~\cite{adams,becker,cls,morarobFan}.
\medskip

Toric ideals form a specific class of ideals which can be defined in different ways~\cite{robbiano}. 
For a subset $A\subset\ZZ^d$ of $n$ vectors or equivalently a matrix $A\in\ZZ^{d\times n}$,
the \textit{toric ideal} $I_A$ is defined as
\begin{align}
I_A=\left\langle \bx^u-\bx^v\mid Au=Av,\,u,v\in\NN_0^n\right\rangle.\label{eq-toricrep1}
\end{align}
Each element $u\in\ZZ^n$ can be uniquely written as $u=u^+-u^-$ where $u^+,u^-$ have disjoint support and their entries are
non-negative. Based on this, the toric ideal $I_A$ can also be expressed as~\cite{sturmfels}
\begin{align}
I_A=\left\langle \bx^{u^+}-\bx^{u^-}\mid u\in\ker(A)\right\rangle.\label{eq-toricrep2}
\end{align}
The binomials in the generating set~(\ref{eq-toricrep2}) are \textit{pure}, i.e., the greatest common divisor of 
the terms $\bx^{u^+}$ and $\bx^{u^-}$
in the binomial $\bx^{u^+}-\bx^{u^-}$ is $1$.


\section{Linear Codes over Prime Fields}
Let $\F$ be a finite field and let $n$ and $k$ be positive integers with $n\geq k$. 
A \textit{linear code\/} of length $n$ and dimension $k$ over $\F$ is the image $\Co$ of a one-to-one linear mapping 
$\phi:\F^k\rightarrow\F^n$, i.e.,  $\Co = \{\phi(a)\mid a\in\F^k\}$. Such a code is denoted as {\em $[n,k]$ code\/} and its elements are called \textit{codewords}.
In algebraic coding, the codewords are always written as row vectors.
Alternatively, a code $\Co$ can be described as 
the row space of a matrix $G\in\F^{k\times n}$, whose rows form a basis of $\Co$, and the matrix $G$ is then called a 
\textit{generator matrix\/} for $\Co$. 
Any other generator matrix for $\Co$ can be obtained from a given generator matrix for $\Co$ by multiplying it from the left with a regular matrix.
A code $\Co$ is in \textit{standard form} if it has a generator matrix which is {\em systematic}, 
i.e., $G=\left(I_k\mid M\right)$, where $I_k$ is the $k\times k$ identity matrix. 
Note that a generator matrix for an $[n,k]$ code can contain a zero column.
Such a code can be shortened by deleting this column giving an $[n-1,k]$ code. 
All subsequently considered codes are assumed to have no such zero column.

Two $[n,k]$ codes are {\em equivalent\/} if one can be obtained from the other by a monomial transformation, 
i.e., a linear map given by a monomial matrix, which is a matrix that has in each row and column exactly one non-zero element.
It follows that every linear code is equivalent to a linear code in standard form.  

The dual code $\Co^\perp$ of an $[n,k]$ code~$\Co$ over $\F$ is an $[n,n-k]$ code consisting of all words $u\in\F^n$ 
such that $u\cdot c=uc^T=0$ for each $c\in\Co$,
where $c^T$ denotes the transposed of $c$.
If $G = \left(I_k\mid M\right)$ is a generator matrix for $\mathcal{C}$, then $H = \left(-M^T\mid I_{n-k}\right)$ is a generator matrix 
for $\mathcal{C}^\perp$. For each word $c\in\F^n$, we have $c\in \mathcal{C}$ if and only if $ Hc^T = \mathbf{0}$.
The matrix $H$ is a \textit{parity check matrix\/} for $\mathcal{C}$.

The \textit{support\/} of a vector $u\in\F^n$, denoted by $\supp(u)$, is the subset of $\underline n=\{1,\ldots,n\}$ given by all indices $i\in\underline{n}$ with $u_i\neq 0$, 
and the \textit{Hamming weight\/}, denoted by $\wt(u)$, is the number of non-zero components and so equals the cardinality of the codeword's support. 
Note that for a binary code, each codeword is completely determined by its support.
The \textit{weight distribution\/} of an $[n,k]$ code $\Co$ is a finite sequence of integers $A_0,A_1,\dots,A_n$, 
where $A_i$, $0\leq i\leq n$, denotes the number of codewords in $\Co$ having Hamming weight~$i$.
The \textit{Hamming distance\/} between two vectors $u,v\in\F^n$ is the number of positions in which they differ and so is given by the Hamming weight $\wt(u-v)$ of the difference vector.
The Hamming distance defines a metric on $\F^n$.
The minimum Hamming distance between any to distinct codewords in $\Co$ is the {\em minimum distance\/} of $\Co$.

For any matrix $G\in\F^{k\times n}$ and any subset $J\subseteq\underline{n}$ of indices, 
let $G_J$ denote the $k\times |J|$ submatrix of $G$ consisting of the columns with indices in $J$.
Similarly, let $c_J$ be the vector of length $|J|$ consisting of the coordinates of $c$ with indices in $J$.
A subset $J\subseteq\underline{n}$ of cardinality $k$ is called an \textit{information set} of the code if the $k\times k$ submatrix $G_J$ has rank $k$.
In particular, the following are equivalent:
\begin{enumerate}
 \item The set of indices $J$ is an information set.
 \item For each $m\in\F^k$ there is a unique $c\in\Co$ with $c_J=m$.
  \item For every generator matrix $G$ of the code $\Co$, $G_J$ has rank $k$.
\end{enumerate}
By the second assertion, a code cannot contain an information set $J\subseteq \underline{n}\setminus\supp(c)$ at the zero positions of a non-zero codeword $c$.

More basics on linear codes can be found in~\cite{macws, vlint}.

For a given $[n,k]$ code $\Co$ over a field $\F_p$ with $p$ elements, define the associated \textit{code ideal\/} as
\begin{align} \label{eq-codeideal}
I_\Co=\left\langle\bx^c-\bx^{c'}\left|\right.c-c'\in\Co\right\rangle+ I_p,
\end{align}
where
\begin{align}
I_p=\left\langle x_i^p-1\left|\right.1\leq i\leq n\right\rangle.
\end{align}
Note that $I_p$ allows to view the exponents of the monomials as vectors in~$\F_p^n$. 
Each codeword $c-c'\in\Co$ with $\supp(c)\cap\supp(c')=\emptyset$
can be associated with the binomial $\bx^{c}-\bx^{c'}$ in the code ideal $I_\Co$, 
where it is always assumed that $\bx^{c}\succ\bx^{c'}$ whenever a monomial order $\succ$ is considered. 
Note that unlike for toric ideals this binomial representation is not unique.
Nevertheless, a binomial $\bx^{c}-\bx^{c'}$ in~$I_\Co$ is said to be associated with the codeword $c-c'$.
Observe that 
the code ideal of a code $\Co$ can be based on a toric ideal as follows,
\begin{align}
I_\Co=I_A+I_p,
\end{align}
where $A$ in an integral $n-k\times n$ matrix such that $H=A\otimes_\ZZ \F_p$ is a parity check matrix for $\Co$.

\section{Universal Gröbner Bases}

In~\cite{mahkhz4} it has be shown that the reduced Gröbner basis for the code ideal~$I_\Co$ with respect to
the lexicographic order can be read off from a generator matrix for the corresponding code~$\Co$.
\begin{theorem}[\cite{mahkhz4}]\label{thm-mainmeh}
Let $\Co$ be an $[n,k]$ code over a prime field\/ $\F_p$ generated by a matrix $G=\left(I_k\mid \ast\right)$ with row vectors 
$g_i=e_i+m_i$, $1\leq i\leq k$, where $e_i$ denotes the $i$th unit vector and $m_i$ a row vector of length $n$
whose first $k$ entries are zero.
The reduced Gröbner basis with respect to any lexicographic order with 
$\{x_1,\dots,x_k\}\succ\{x_{k+1},\dots,x_n\}$ for the code ideal $I_\Co$ is given by
\begin{align*}
\mathcal{G}_\succ\left(I_\Co\right)=\left\{x_i-\bx^{m_i}\mid 1\leq i\leq k\right\}\cup\left\{x_i^p-1\mid k+1\leq i\leq n\right\}.
\end{align*}
\end{theorem}
This result can be further generalized.
\begin{corollary}\label{cor-generalized}
Let $\Co$ be an $[n,k]$ code over a prime field $\F_p$ with an information set $J=\{i_1,\dots,i_k\}\subseteq\underline{n}$. 
There exists a generator matrix $G$ in reduced row echelon form with respect to the columns indexed by $J$ and row vectors
$g_{i_j}=e_{i_j}+m_j$, where $1\leq j\leq k$, and
the reduced Gröbner basis with respect to any lexicographic order with 
$\{x_j\mid j\in J\}\succ\{x_\ell,\mid \ell\in \underline{n}\setminus J\}$ for the code ideal $I_\Co$ is given by
\begin{align*}
\mathcal{G}_\succ\left(I_\Co\right)=\left\{x_{i_j}-\bx^{m_j}\mid i_j\in J\right\}\cup\left\{x_\ell^p-1\mid \ell\in \underline{n}\setminus J\right\}.
\end{align*}
\end{corollary}
In~\cite{sturmfels} the author has introduced several concepts in the context of toric ideals which will be utilized in the following. 
However, since code ideals are not toric but a sum of a toric and a non-prime ideal, 
several of these concepts need to be adapted. 
In particular, it will become apparent that binary and non-binary codes need to be distinguished.

A binomial~$\bx^{c}-\bx^{c'}$ in $I_\Co$ is called \textit{primitive} if there is no other binomial $\bx^{u}-\bx^{u'}$ in $I_\Co$
such that $\bx^{u}$ divides $\bx^{c}$ and $\bx^{u'}$ divides $\bx^{c'}$. 
Additionally, every binomial $x_i^p-1$ for $1\leq i\leq n$ is considered to be primitive. For binary codes, the case
of $c'=\mathbf{0}$ is excluded for reasons which will later become apparent. 
Note that if a monomial order $\prec$ is given, $c=\mathbf{0}$ implies $\bx^{c}-\bx^{c'}=0$ because for each binomial it is assumed that $\bx^{c}\succ\bx^{c'}$. 
The \textit{Graver basis} for $\Co$ consists of all primitive binomials lying in the corresponding code ideal and is denoted by~$\Gr_\Co$.

For a binary code $\Co$, a codeword $c$ in $\Co$ is a \textit{circuit} if its support is minimal with respect to inclusion.
Each binomial associated with such a codeword is also called a circuit. 
It follows that every binomal which is a circuit is also primitive.

In~\cite{sturmfels} circuits were defined by the additional condition that their entries are relatively prime. 
In the binary case this condition can be omitted. 
The extension to codes over an arbitrary prime field $\F_p$, however, cannot simply be accomplished by adding this condition. 
A different definition is required in order to obtain similar results
for the code ideal over such a field (Prop.~\ref{prop-inclusionp})  
since the exponents of the polynomials are treated as vectors in $\F_p^n$. The arising difficulties
in this case are illustrated as follows.

\begin{example}\label{ex-problemFp}
Consider a linear code~$\Co$ over $\F_7$ generated by $$G=\begin{pmatrix}0&1&1\\1&3&0\end{pmatrix}$$
and the code ideal $I_\Co$ in ${\Bbb Q}[a,b,c]$.
The codeword $(1,3,0)$ and all its multiples $(2,6,0)$, $(3,2,0)$, $(4,5,0)$, $(5,1,0)$ and $(6,4,0)$
have minimal support with respect to inclusion and except for $(2,6,0)$ and $(6,4,0)$ their entries are relatively prime. 
However, for the codeword $(3,2,0) = (3,0,0) - (0,5,0)$, the corresponding binomial $a^3-b^5$ is not primitive 
since $a^2-b$ lies in the code ideal, $a^2$ divides $a^3$, and $b$ divides $b^5$. 
It follows that minimal support and relative primeness of entries are not sufficient 
to ensure that such a binomial is primitive. 

Another drawback is that different representations of a codeword yield binomials with distinct attributes. 
Writing $(2,6,0)\!=\!(2,0,0)-(0,1,0)$ yields the primitive binomial $a^2-b$. 
However, expanding $(2,6,0)=(0,6,0)-(5,0,0)$ gives the binomial $b^6-a^5$, 
which is not primitive because the binomial $b^2-a^4$ corresponding to $(3,2,0)=(0,2,0)-(4,0,0)$ also belongs to the code ideal. 
\hfill $\diamondsuit$
\end{example}

Note that for toric ideals, when the exponents are viewed as vectors in $\ZZ^n,$
the condition of relative primeness of the entries together with the minimality of the support is sufficient
to guarantee that the corresponding binomial is primitive since for any $c\in\ZZ^n$ the representation $c=c^+-c^-$ is unique.

Motivated by the above example, a more general definition of circuits for codes over arbitrary prime fields is required:
A binomial $\bx^{c}-\bx^{c'}$ in $I_\Co$ with $c'\neq\mathbf{0}$ is called a \textit{circuit} if it is a primitive binomial and its support is minimial with respect to
inclusion. 
In the non-binary case the attribute of being a circuit is tied to the binomial associated with a codeword rather than
the codeword itself because a codeword may be associated with both, a primitive and a non-primitive binomial as the above example 
has demonstrated.    
With this more general definition, denote by $\text{C}_\Co$ all circuits lying in the code ideal $I_\Co$.
Finally, denote the universal Gröbner basis for the code ideal $I_\Co$ by $\mathcal{U}_\Co$.

\begin{proposition}\label{prop-inclusionp}
For a linear code $\Co$ over $\F_p$, $\text{C}_\Co\subseteq \mathcal{U}_\Co\subseteq\Gr_\Co$.
\end{proposition}
\begin{proof}
The inclusion $\mathcal{U}_\Co\subseteq\Gr_\Co$ is proved in~\cite{mahkhz3}. 
So it remains to show that $\text{C}_\Co\subseteq\mathcal{U}_\Co$.

Let $\bx^{c'}-\bx^{c''}\in I_\Co$ be a circuit corresponding to the codeword $c=c'-c''\in\Co$. 
Put $s=\deg({c'})$ and $t=\deg({c''})$ and choose an elimination order $\succ$ such that 
$\{x_i\mid i\notin\supp(c)\}\succ\{x_i\mid i\in\supp(c)\}$ and the monomials in $\{x_i\mid i\in\supp(c)\}$ are first
compared by their $\omega$-degree, where $\omega_i=t$ whenever $i\in\supp(c')$ and $\omega_i=s$ whenever $i\in\supp(c'')$
and ties are broken by any lexicographic order with $\{x_i\mid i\in\supp(c')\}\succ\{x_i\mid i\in\supp(c'')\}$.
For this order, $\bx^{c'}\succ\bx^{c''}$ because of $c'\cdot\omega=st=c''\cdot\omega$ and the chosen tie breaker.

Claim that $\bx^{c'}-\bx^{c''}\in\mathcal{G}_\succ(I_\Co)$. 
Indeed, since this binomial belongs to $I_\Co$, it must 
be reduced to zero by binomials in $\mathcal{G}_\succ(I_\Co)$ and in particular, there must be a binomial
$\bx^{v'}-\bx^{v''}\in\mathcal{G}_\succ(I_\Co)$ such that $v=v'-v''\in\Co$ and its leading term $\bx^{v'}$ divides $\bx^{c'}$. 
But then $\supp(v')\subseteq\supp(c')$ and by the choice of the monomial order it follows $\supp(v'')\subseteq\supp(c)$. Hence,
$\supp(v)\subseteq\supp(c)$. But as $c$ has minimal support this inclusion cannot be proper and so $\alpha c=v$ for some $\alpha\in\F_p$. 

Two cases occur.
First, consider the case $\alpha=1$, i.e., $c=v$. Then by the choice of monomial order it can be
further deduced that $\bx^{v'}=\bx^{c'}$. Otherwise, $\bx^{c'}$ would contain more variables than $\bx^{v'}$, say 
$\bx^{c'}=\bx^{v'}\bx^{w}$ for some $w\neq\mathbf{0}$ and then these missing variables must appear in the second term, 
$\bx^{v''}=\bx^{c''}\bx^{\mathbf{p}-w}$, where $\mathbf{p}$ is the all-$p$ vector. 
But then
\begin{align*}
v'\cdot\omega&=(s-|w|)t<st+pt-|w|t=st+(p-|w|)t=v''\cdot\omega
\end{align*}
which contradicts $\bx^{v'}\succ\bx^{v''}$. 

Second, consider the case $\alpha\neq 1$. Because of $\supp(v')\subseteq\supp(c')$ either $\supp(v')\subsetneq\supp(c')$
or $\supp(v')=\supp(c')$.

First assume that $\supp(v')=\supp(c')$ and thus $\supp(v'')\!=\!\supp(c'')$. 
Since $\bx^{v'}$ divides $\bx^{c'}$ the monomial $\bx^{v''}$ cannot divide $\bx^{c''}$ for $\bx^{c'}-\bx^{c''}$ is primitive. 
But as $\supp(v'')=\supp(c'')$ 
the degree of $\bx^{v''}$ must be strictly greater than that of $\bx^{c''}$. Hence,
$c''\cdot\omega<v''\cdot\omega$ because all $x_i$ with $i\in\supp(c'')$ are weighted equally.
Furthermore, as $\supp(v')=\supp(c')$ and $\bx^{v'}$ divides $\bx^{c'}$ the same argument yields
$v'\cdot\omega<c'\cdot\omega$.
It follows that $v'\cdot\omega<c'\cdot\omega=st=c''\cdot\omega<v''\cdot\omega$ 
contradicting the relation $\bx^{v'}\succ \bx^{v''}$.

Second assume that $\supp(v')\subsetneq\supp(c')$.
Here the same inequality can be established when variables are shifted from $\bx^{v'}$ to $\bx^{v''}$ 
because all entries in~$\omega$ are positive. 
This will also lead to the contradiction that $\bx^{v'}\succ\bx^{v''}$. 

In view of the two cases, it follows that $\alpha=1$.
By the first case this means that $\bx^{v'}=\bx^{c'}$ and therefore $\bx^{c'}-\bx^{c''}=\bx^{v'}-\bx^{v''}$, as required.
\end{proof}

The proof justifies that a binomial of the form $\bx^{c}-1$ is not considered as a circuit.
Indeed, if $\bx^{c}-1$ were the circuit considered in the proof, 
the weight vector $\omega$ introduced there would be $\omega=\mathbf{0}$ 
and so the contradiction $v'\cdot\omega <v''\cdot\omega$ could not be achieved. 

This will be underpinned by the next example, which will show for a specific code 
that there exist primitive binomials $\bx^c-1$ such that $c$ has minimal support and does not belong to the universal Gröbner basis.

For binary codes, however, it can be shown that a binomial of the form $\bx^c-1$ cannot belong to the universal Gröbner basis.
To see this, assume that $\bx^c-1$ lies in some reduced Gröbner basis for an arbitrary monomial order.
Since the basis is reduced, the binomials of the form $\bx^{c'}-\bx^{c''}$, where $c=c'-c''$ and $c''\ne \mathbf{0}$, cannot belong to this Gröbner basis.
But for any such binomial $\bx^{c'}-\bx^{c''}$ there must be a binomial in the reduced Gröbner basis whose leading term divides the leading term of $\bx^{c'}-\bx^{c''}$.
This binomial will then also divide the leading term of $\bx^c-1$ contradicting the reducedness of the basis.
\medskip

In view of binary linear codes, it will be shown that all three sets coincide. 
For non-binary linear codes, however, the next example will illustrate that the inclusions 
can be strict.
\begin{example}\label{ex-2}
In view of the code~$\Co$ from Ex.~\ref{ex-problemFp}, 
computations using the software package {\tt Gfan}~\cite{gfan} exhibits that the set of circuits is 
\begin{align*}
\text{C}_\Co&=\{b-c^6, a-c^3, c^6-b,b^2-c^5,c^3-a,a^2-b,b-a^2,a^3-c^2,\\
&\quad\quad c^2-a^3,a^5-c,c-a^5,c^5-b^2,b^3-c^4,c^4-b^3,b^4-c^3,\\
&\quad\quad b^4-a,c^3-b^4,b^5-c^2,a-b^4,c^2-b^5, b^6-c,c-b^6\}
\end{align*}
and that the universal Gröbner basis for the code ideal $I_\Co$ is indeed a proper superset,
\begin{align*}
\mathcal{U}_\Co= \text{C}_\Co\cup
\{bc-1,a^2c-1,ab^3-1,b^2-ac^2,ab-c^2, b^3-ac, ac^2-b^2,\\
 c^2-ab,ab^2-c,ac-b^3,c-ab^2\}\cup\{a^7-1,b^7-1,c^7-1\}
\end{align*}
Moreover, the universal Gröbner basis $\mathcal{U}_\Co$ properly lies inside the Graver basis $\Gr_\Co$ 
since $ac^4-1$ is a primitive binomial that belongs to $\Gr_\Co$ but not to $\mathcal{U}_\Co$.
To see this, note that the binomial $ac^4-1$ corresponds to the codeword $(1,0,4) = (1,0,4)-(0,0,0)$ and only a
binomial of the form $a^sc^t-1$ with either $s<1$ and $t\leq 4$ or $s\leq1$ and $t<4$ could contradict its being
primitive. But clearly no such codeword exists.
Note additionally that the corresponding codeword $(1,0,4)$ has minimal support.
\hfill $\diamondsuit$
\end{example}

In the following, binary linear codes will only be considered.

\begin{theorem}[\cite{mahkhz3}]\label{thm-binequal}
For a binary linear code~$\Co$ the set of circuits $C_\Co$ equals the Graver basis $Gr_\Co$.
\end{theorem}

Combining Thm.~\ref{thm-binequal} and Prop.~\ref{prop-inclusionp} yields the following important result.
\begin{corollary}\label{cor-uniGBallprim}
For a binary linear code~$\Co$ the universal Gröbner basis $\mathcal{U}_\Co$ of the corresponding code ideal $I_\Co$
consists of all primitive binomials.
\end{corollary}

In~\cite{sturmfels} the author has shown that the total degree of any primitive binomial in the toric ideal $I_A$ is bounded by
$(n-d)(d+1)D(A)$, where $d\times n$ is the size of the matrix $A$ and $D(A)$ is an integer number depending only on the entries in $A$. 
This result makes use of the estimate $|u_i|\leq D(A)$, $1\leq i\leq n$, for any
circuit $u=(u_1,\dots,u_n)\in \ker(A)$. 
Furthermore, the author has conjectured that an even better estimation holds, 
namely that the total degree is bounded by $(d+1)D(A)$. In the notation of binary linear codes the row size is $d=n-k$
and any entry of a codeword is either $0$ or $1$ and thus $D(A)$ can be chosen to be~1. 
Hence, the proven estimate becomes $k(n-k+1)$.  
However, it will be shown that the bound $n-k+1$ conjectured by the author indeed holds.
Note that this bound corresponds to the Singleton bound on the minimum distance of linear codes which is attained with equality by the maximum distance separable (MDS) codes 
like the Reed-Solomon codes and their extended versions. 
In the binary case, only trivial MDS codes exist~\cite{macws,vlint}.

\begin{proposition}\label{prop-primcode}
Let $\Co$ be a binary $[n,k]$ code. 
If $\bx^{c}-\bx^{c'}\in I_\Co$ is primitive, 
then $\wt(c-c')\leq n-k+1$ and for any generator matrix $G$ of the code $\Co$ the submatrix
$G_{\underline{n}\setminus\supp(c-c')}$ has rank $k-1$.
\end{proposition}
\begin{proof}
According to Thm.~\ref{thm-binequal} the primitive binomials of the code ideal $I_{\Co}$ are exactly the circuits. 
The latter are given by sets of minimally dependent column vectors of a parity check matrix for the code. 
Any parity check matrix is of size $(n-k)\times n$ and has rank $n-k$. 
Such a matrix has at most $n-k$ linearly independent columns, 
which implies that the Hamming weight of a circuit is at most $n-k+1$.

To show the second assertion, let $c\in\Co$ be a circuit. 
As $c$ is a codeword there is a non-zero information word $x\in\mathbb{F}_2^k$ with $x\cdot G=c$ for any generator matrix $G$. 
But $x\cdot G_{\underline{n}\setminus\supp(c)}=\mathbf{0}$ and so the matrix $G_{\underline{n}\setminus\supp(c)}$ cannot have maximal rank $k$.
Suppose the rank of $G_{\underline{n}\setminus\supp(c)}$ is smaller than $k-1$. 
By the dimension formula for linear maps,
\begin{align*}
k=\dim \ker G_{\underline{n}\setminus\supp(c)} + \dim {\rm im}\, G_{\underline{n}\setminus\supp(c)}<\dim \ker G_{\underline{n}\setminus\supp(c)}+(k-1)
\end{align*}
and so $ \dim \ker G_{\underline{n}\setminus\supp(c)}>1.$
Thus there must be another information word $x'\in\F_2^k$ with $x'\cdot G_{\underline{n}\setminus\supp(c)}=\mathbf{0}$.
Put $c' = x'\cdot G$.
So for each index $i$ in $\underline{n}\setminus\supp(c)$,  $c'_i=x'\cdot G_{{\{i \}}}=0$ and thus $\supp(c')\subseteq\supp(c)$.
But the encoding is one-to-one and so the codeword $c'$ is distinct from $c$. 
It follows that $\supp(c')\subsetneq\supp(c)$ contradicting the hypothesis that $c$ is a circuit.
Hence the rank of $G_{\underline{n}\setminus\supp(c)}$ must be equal to $k-1$.
\end{proof}

The converse of this assertion also holds.
\begin{proposition}\label{prop-codeprim}
Let $\Co$ be a binary $[n,k]$ code with generator matrix $G$. 
Every binomial in the code ideal $I_\Co$ associated with a codeword $c$ of Hamming weight less than or equal to $n-k+1$ and
such that $G_{\underline{n}\setminus \supp(c)}$ has rank $k-1$ is primitive. 
\end{proposition}
\begin{proof}
Consider a codeword $c\in\Co$ with Hamming weight $\leq n-k+1$ and
such that $G_{\underline{n}\setminus \supp(c)}$ has rank $k-1$. 
Then $c$ has at least $k-1$ entries that are zero and so by hypothesis, among those one can find exactly $k-1$ coordinates 
$J\subseteq\underline{n}\setminus\supp(c)$ such that $G_J$ has rank $k-1$. 
But as the generator matrix~$G$ has rank $k$ there must be another column in $G$, say indexed by~$i$, 
with $c_i=1$, and such that $G_{J\cup{\{i\}}}$ is a $k\times k$ matrix of rank $k$; 
that is, $J\cup{\{i\}}$ is an information set. 
By Cor.~\ref{cor-generalized}, the binomial $x_i-\bx^{c-e_i}$
belongs to the reduced Gröbner basis for a particular lexicographic order. 
Since every binomial in a reduced Gröbner basis for $I_\Co$ is primitive (see Prop.~\ref{prop-inclusionp}), the binomial $x_i-\bx^{c-e_i}$ is primitive.
But if $\bx^u-\bx^{u'}$ is primitive, then any other binomial $\bx^{v}-\bx^{v'}$ with $u-u'=v-v'$ is also primitive. 
Hence, each binomial associated with the codeword $c$ is primitive, too.
\end{proof}

\begin{theorem}\label{thm-universalbasis}
 Let $\Co$ be a binary $[n,k]$ code. 
The universal Gröbner basis for the corresponding code ideal $I_\Co$ is given by the set
\begin{align*}
 \mathcal{U}_\Co&=\big\{\bx^{c}-\bx^{c'}\left|\right. c-c'\in\Co, \wt(c-c')\leq n-k+1,\\
 &\quad\quad\rank\left(G_{\underline{n}\setminus\supp(c-c')}\right)=k-1 \big\}
\cup\left\{x_i^2-1\mid 1\leq i\leq n \right\}.
\end{align*}
In other words, the universal Gröbner basis for the code ideal consists of all binomials which correspond to codewords that 
satisfy the Singleton bound and a particular rank condition. 
\end{theorem}
\begin{proof}
Prop.~\ref{prop-primcode} and~\ref{prop-codeprim} state that a binomial $\bx^c-\bx^{c'}\in I_\Co$ is primitive if and
only if it corresponds to a codeword of Hamming weight $\leq n-k+1$ and such that the submatrix $G_{\underline{n}\setminus\supp(c-c')}$
 of any generator matrix $G$ has rank $k-1$. 
The result follows by applying Cor.~\ref{cor-uniGBallprim}.
\end{proof}

This result gives rise to a new class of binary linear codes whose codewords which fulfill the Singleton bound also satisfy the rank condition.
A binary linear code $\Co$ is called a {\em Singleton code\/} if each non-zero codeword $c$ with Hamming weight $\leq n-k+1$ has the property that 
the submatrix $G_{\underline{n}\setminus \supp(c)}$ has rank $k-1$ for any generator matrix $G$ for~$\Co$. 

If $\Co$ is a Singleton code, then by Thm.~\ref{thm-universalbasis} the corresponding universal Gröbner basis 
can be combinatorially constructed.
For this, note that if $c$ is a codeword with Hamming weight $s$, then there are $2^s-2$ binomials associated with $c$.

\begin{example}
The third binary Hamming code $\Co$ is a $[7,4]$ code with weight distribution $1,0,0,7,7,0,0,1$. 
By inspection, this is a Singleton code, i.e., for any codeword $c$ of Hamming weight $\leq 4$ holds
$\rank(G_{\underline{7}\setminus \supp(c)}) \leq 3$ 
for any generator matrix $G$ for $\Co$.
Thus the universal Gröbner basis $\mathcal{U}_\Co$ consists of all binomials which correspond to the codewords with Hamming weight of at most~4.

Computations using {\tt Gfan}~\cite{gfan} exhibit that the universal Gröbner basis consists of~$147$ binomials 
given by seven binomials of the form $x_i^2-1$, $1\leq i\leq 7$, 
$42=7\cdot (2^3-2)$ binomials corresponding to the seven codewords of Hamming weight~$3$, and 
$98 = 7\cdot (2^4-2)$ binomials associated with the seven codewords of Hamming weight~$4$. 
\hfill $\diamondsuit$
\end{example}

Singleton codes are the parity check codes, 
the MDS codes,
the binary Golay code and its parity check extension, 
the Simplex codes, and the first order Reed-Muller code and its dual. 
On the other hand, not all Hamming codes are Singleton ones.
We will provide more details in an upcoming paper.\medskip

As a final remark, the authors in~\cite{martinez1} have introduced a method for computing the Graver basis for a linear code $\Co$ over $\ZZ_p$, where $p\geq 2$ is an integer, 
which amounts to computing the Gröbner basis of a certain ideal. 
Since in the binary case the Graver basis coincides with the universal Gröbner basis, this provides another method for computing the universal Gröbner basis for a binary linear code.


\bibliographystyle{plain}
\bibliography{refBook,refPaper}

\end{document}